\documentclass[a4paper,11pt]{amsart}

\usepackage{graphicx}
\usepackage{amsmath}
\usepackage{amsthm}
\usepackage{amssymb}
\usepackage{verbatim}
\usepackage{epsfig}
\usepackage{color}
\usepackage{url}
\usepackage{pinlabel}

\newtheorem{thm}{Theorem}[section]    
\newtheorem{lemma}[thm]{Lemma}
\newtheorem{cor}[thm]{Corollary}          

\theoremstyle{definition}
\newtheorem{eg}[thm]{Example}                                
\newtheorem{question}[thm]{Question}

\newcommand{\ul}[1]{\ensuremath{\underline{\underline{#1}}}}

\newcommand{\bd}{\ensuremath{\partial}}

\newcommand{\co}{\colon\thinspace}
\newcommand{\RR}{\mathbf{R}}

\newcommand{\ZZ}{\mathbf{Z}}
\newcommand{\torus}{\mathrm{T}}
\newcommand{\sphere}{\mathrm{S}}
\newcommand{\disk}{\mathrm{D}}
\newcommand{\id}{\mathrm{id}}

\DeclareMathOperator{\closure}{Cl}

\author{Bell Foozwell}
\address{Trinity College, The University of Melbourne}
\email{bfoozwel@trinity.unimelb.edu.au}
\urladdr{https://sites.google.com/site/bellfoozwell/}
\author{Hyam Rubinstein}
\address{Department of Mathematics and Statistics, The University of Melbourne}
\email{rubin@ms.unimelb.edu.au}
\urladdr{http://www.ms.unimelb.edu.au/~rubin/}
\title{Four-dimensional Haken cobordism theory.}
\date{\today}
\begin{document}

\begin{abstract}

Cobordism of Haken $n$--manifolds is defined by a Haken $(n+1)$--manifold $W$ whose boundary has two components, each of which is a closed Haken $n$--manifold. In addition, the inclusion map of the fundamental group of each boundary component to $\pi_1(W)$ is injective. In this paper we prove that there are $4$--dimensional Haken cobordisms whose boundary consists of any two closed Haken $3$--manifolds. In particular, each closed Haken $3$--manifold is the $\pi_1$--injective boundary of some Haken $4$--manifold.
\end{abstract}

\maketitle

\section{Introduction}\label{Introduction}
The authors have defined and studied Haken $n$--manifolds and Haken cobordism theory in previous work \cite{FoozRubin2011}. These manifolds enjoy important properties --- for example, the universal cover of a closed Haken $n$--manifold is $\RR^n$ (see Foozwell \cite{Fooz1}). 
We would like to know if Haken $4$--manifolds are abundant or relatively rare manifolds. We will show that they are abundant in the following sense:\\
\emph{For each pair of closed Haken $3$--manifolds $M, M^\prime$, there is a Haken $4$--manifold $W$ with boundary $\bd W = M \cup M^\prime$. In addition, the inclusion maps induce injections $\pi_1(M) \to \pi_1(W)$ and $\pi_1(M^\prime) \to \pi_1(W)$. The special case when $M^\prime = \varnothing$ is of particular interest.}

Our proof of this result will be obtained in a number of steps. The first step is to show that if $M$ is a torus-bundle over a circle, then there is a Haken $4$--manifold $W$ with boundary $\bd W =M$. We do this in section~\ref{S:Torus-bundles}. We then show a similar result for general surface-bundles in section~\ref{S:Higher genus surface-bundles}. To show that Haken manifolds satisfy our main result, we use a result of Gabai \cite{Gabai} and Ni \cite{Ni} in section~\ref{S:Other Haken manifolds}.

It is well-known that all closed $3$--manifolds are null cobordant, i.e bound compact $4$--manifolds. Davis, Januszkiewicz and Weinberger \cite{DavisJanWein} following on from work in 
\cite{DavisJan}, show that if an aspherical closed $n$--manifold is null cobordant, then it bounds an aspherical $(n+1)$--manifold, and furthermore, the inclusion map of the boundary is $\pi_1$--injective.  Haken $n$--manifolds satisfy the stronger property (than asphericity) that they have universal covering by $\RR^n$, as shown in  \cite{Fooz1}. Moreover for Haken cobordism theory (see \cite{FoozRubin2011}), the inclusion maps of the $n$--manifolds into the $(n+1)$--dimensional cobordism are $\pi_1$--injective. 

\section{Acknowledgements}
We wish to thank Marcel B\"{o}kstedt, Allan Edmonds, Alan Reid and Shmuel Weinberger for helpful comments. The second author would like to thank the Center for Mathematical Sciences of Tsinghua University  for hospitality during part of this research. 
 
The second author was partially supported by the Australian Research Council.

\section{Haken $n$--manifolds}

For simplicity, all manifolds will be assumed to be orientable throughout this paper. 

Let $W$ be a compact $n$--manifold and let $\ul w$ be a finite collection of $(n-1)$--dimensional submanifolds in $\bd W$. We say that $\ul w$ is a \emph{boundary-pattern} if whenever $A_1, \dots, A_i$ is a collection of distinct elements of $\ul w$, then $A_1 \cap \dots \cap A_i$ is an $(n - i)$--dimensional manifold\footnote{The only manifold of negative dimension is the empty set. The empty set is also a manifold in each non-negative dimension.}. A boundary-pattern is \emph{complete} if $\bd W = \cup\{A : A \in \ul w\}$.

The empty boundary-pattern is a special case of a boundary-pattern, and thus a closed manifold is a manifold with boundary-pattern.

Boundary-patterns arise naturally in splitting situations. Suppose that $M$ is a two-sided codimension-one submanifold of $W$. Let $W \vert M$ denote the manifold obtained by splitting $W$ open along $M$. There is a surjective map $q \co W\vert M \to M$, that reverses the process of splitting $W$ open along $M$. We call $q$ the \emph{unsplitting map}. If $W$ has a boundary-pattern $\ul w$, then $B$ is an element of the \emph{natural boundary-pattern} of $W \vert M$ if either
\begin{itemize}
  \item $B$  is a component of $q^{-1}(A)$ for some  $A \in \ul w$, or 
  \item $B$  is a component of $q^{-1}(M)$.
\end{itemize}

Suppose that $(W, \ul w)$ is a manifold with boundary-pattern, that $g : \disk^2 \to \bd W$ is a piecewise-linear map and that the image of $g$ is contained in at most three elements of $\ul w$. If $g(\disk^2)$ is contained in a single element of $\ul w$, say $A_1$, then $g^{-1}(A_2)$ is $\disk^2$ and we say that the \emph{preimage of $g$ is a type 1 disk}. If $g(\disk^2)$ is contained in two distinct elements of $\ul w$, say $A_1$ and $A_2$, then $g^{-1}(A_1) \cup g^{-1}(A_2) = \disk^2$. If both  $g^{-1}(A_1)$  and $g^{-1}(A_2)$ are disks, then we say that the \emph{preimage of $g$ is a type 2 disk}. If $g(\disk^2)$ is contained in three distinct elements of $\ul w$, say $A_1, A_2$ and $A_3$, then $g^{-1}(A_1) \cup g^{-1}(A_2) \cup g^{-1}(A_3) = \disk^2$. If each of $g^{-1}(A_1)$, $g^{-1}(A_2)$ and $g^{-1}(A_3)$ are disks, then we say that the \emph{preimage of $g$ is a type 3 disk}. Type $j$ disks are illustrated in figure~\ref{F:Subdivided disks} for $j \in \{1,2,3\}$.
\begin{figure}[ht!]
\labellist
\small\hair 2pt
\pinlabel $A_1$ at 37 39
\pinlabel $A_1$ at 102 39
\pinlabel $A_2$ at 140 39
\pinlabel $A_1$ at 189 43
\pinlabel $A_2$ at 223 43
\pinlabel $A_3$ at 208 19
\endlabellist
\centering
 \includegraphics{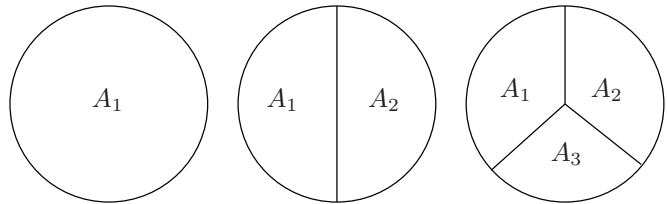}
	\caption{Disks of types 1, 2 and 3.}\label{F:Subdivided disks}
\end{figure}

Let $j \in \{1,2,3\}$ and let $f \co \disk^2 \to W$ be a map such that $f(\bd \disk^2)$ is a loop in $\bd W$ that meets $j$ distinct elements of the boundary-pattern. Suppose that for each such map $f$ there is a map $g \co \disk^2 \to \bd W$ homotopic to $f \text{ rel } \bd $ such that the preimage of $g$ is a type $j$ disk.
 Then we say that $\ul w$ is a \emph{useful} boundary-pattern.
 
 In his solution to the word problem, Waldhausen \cite{Wald} showed that the boundary-patterns that arise in splitting situations for Haken $3$--manifolds can always be modified to be useful. (Note that boundary patterns were formally introduced later by Johannson in \cite{Johannson} --- they were not explicitly mentioned in \cite{Wald}).

A map between manifolds with boundary-patterns should relate the boundary-patterns in a reasonable way. We use the following definition. If $(W,\ul w)$ and $(V, \ul v)$ are manifolds with boundary-patterns, then an \emph{admissible map} is a continuous function $f \co W \to V$ that is transverse to the boundary-patterns and satisfies
\begin{align*}
 \ul w = \bigsqcup_{A \in \ul v} \{B : B \text{ is a component of } f^{-1}(A) \}.
\end{align*}
We write $f \co (W, \ul w) \to (V, \ul v)$ to indicate that the map $f$ is admissible. Admissible homeomorphisms, embeddings and so on are defined in the obvious way. 

If a properly embedded arc can be pushed into the boundary-pattern so that it is contained in no more than two boundary-pattern elements, then we say that the arc is \emph{inessential}. More precisely, let $(J, \ul j)$ be a compact $1$--dimensional manifold with complete boundary-pattern and let $\sigma \co (J, \ul j) \to (W, \ul w)$ be an admissible map. We say that $\sigma$ is an \emph{inessential curve} if there is an admissible map $g \co (\Delta, \ul \delta) \to (W, \ul w)$ such that
 \begin{enumerate}
 	\item $J = \closure \left(\bd \Delta \setminus \bigcup\{A : A \in \ul \delta \} \right)$
 	\item $\ul \delta$ consists of at most two elements,
 	\item $g \vert_J = \sigma$.
 \end{enumerate} 
The boundary-pattern $\ul \delta$ consists of one element if both endpoints of $\sigma$ are contained in the same element of $\ul w$, consists of two elements if the endpoints of $\sigma$ are contained in distinct elements of $\ul w$.  If $J$ is a circle, then $\ul \delta$ is empty. We say that $\sigma \co (J, \ul j) \to (W, \ul w)$ is an \emph{essential curve} if there is no map $g\co (\Delta, \ul \delta) \to (W, \ul w)$ satisfying the three properties above.

An admissible map $f \co (M, \ul m) \to (W, \ul w)$ is \emph{essential} if each essential curve $\sigma \co (J, \ul j) \to (M, \ul m)$ defines an essential curve $f \circ \sigma \co (J, \ul j) \to (W, \ul w)$. Let $M$ be a submanifold of $W$. We say that $(M, \ul m)$ is an \emph{essential submanifold} of $(W, \ul w)$ if the inclusion map is admissible and essential.

Let $(W, \ul w)$ be an $n$--manifold with complete and useful boundary-pattern and let $(M, \ul m)$ be a two-sided codimension-one submanifold of $W$ for which the inclusion map is admissible and essential. Then we say that $(W,M)$ is a \emph{good pair}.

A \emph{Haken $1$--cell} is an arc with complete and useful boundary-pattern. If $n >1$, then a \emph{Haken $n$--cell} is an $n$--cell with complete and useful boundary-pattern such that each element of the boundary-pattern is a Haken $(n-1)$--cell.

Let $(W_0, \ul w_0)$ be an $n$--manifold with complete and useful boundary-pattern. A finite sequence of good pairs
\begin{align*}
  (W_0, M_0), (W_1,M_1), \dots, (W_k, M_k)
\end{align*}
is called a \emph{hierarchy} if
\begin{enumerate}
	\item $W_{i+1}$ is obtained by splitting $W_i$ open along $M_i$,
	\item $W_{k+1}$ is a finite disjoint union of Haken $n$--cells.
\end{enumerate}	 
A manifold with a hierarchy is called a \emph{Haken $n$--manifold}.

By definition, each element of the boundary-pattern of a Haken $n$--manifold is $\pi_1$--injective. By convention, when we say that a manifold is Haken without explicitly referring to a boundary-pattern, the boundary-pattern is simply the disjoint union of the boundary components. For example, suppose that a manifold $W$ has two boundary components, $X$ and $Y$. If we assert that $W$ is a Haken manifold, then this is meant to imply that $X$ and $Y$ are $\pi_1$--injective in $W$ and that the boundary-pattern of $W$ is $\{X, Y\}$.

Fibre-bundles that have aspherical surfaces as base and fibre provide examples of Haken $4$--manifolds. The hierarchy is obtained by lifting essential curves and arcs in the base surface to the $4$--manifold. These manifolds will play an important role in this paper.

Let $\ul w = \{M_1, \dots, M_j\}$ be a finite collection of closed Haken $n$--manifolds. If $W$ is a connected Haken $(n+1)$--manifold with boundary-pattern $\ul w$, then we say that $W$ is a \emph{Haken cobordism}. If the collection $\ul w$ consists of just two manifolds, then we may regard a Haken cobordism as an equivalence relation between Haken $n$--manifolds.

Our interest is in Haken cobordism as a relation between Haken $3$--manifolds. In section~\ref{S:Other Haken manifolds}, we will give a condition for two connected Haken $3$--manifolds to form the boundary of a Haken cobordism. We will also show that each closed Haken $3$--manifold is the boundary of some Haken $4$--manifold. As a first step, the following lemma was proved in Foozwell's thesis \cite{Fooz}.
\begin{lemma}\label{L: Re-gluing}
If $N$ is obtained from the Haken $3$--manifold $M$ by splitting $M$ open along an incompressible surface $F$ and re-gluing the boundary components, then there is a Haken $4$--manifold $W$ with $\bd W = M \sqcup N \sqcup E$, where $E$ is a surface-bundle over the circle with fibre $F$.
\end{lemma}
The idea of the proof is to form $M \times [0,1]$ and attach a copy of $R(F) \times [0,1]$ to a regular neighbourhood of parallel copies of $F \times \{1\}$ in $M \times \{1\}$ as indicated in figure~\ref{F:regluing_fig}. We denote by $R(F)$ the regular neighbourhood of $F$. It is easy to see that the right boundary components are obtained, and with a little more work\footnote{The extra work consists of showing that each boundary component is $\pi_1$--injective, and showing that a hierarchy exists.} one can see that $W$ is indeed a Haken $4$--manifold.  
\begin{figure}[ht!]
\labellist
\small\hair 2pt
\pinlabel {$R(F) \times [0,1]$} at 163 97
\pinlabel {$M \times \{0\}$} [tr] at 200 1
\pinlabel {$M \times \{1\}$} [tr] at 200 78
\pinlabel $F$ [tr] at 103 1
\endlabellist
\centering
	\includegraphics{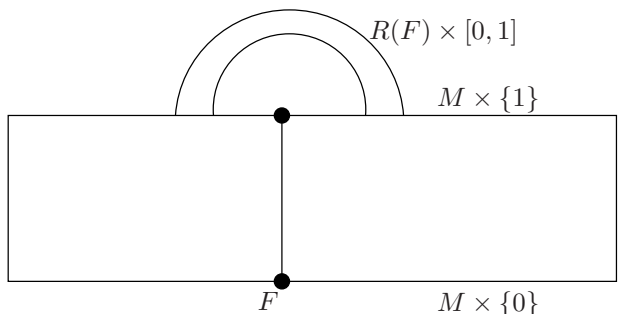}
	\caption{Building a Haken $4$--manifold with three boundary components.}\label{F:regluing_fig}
\end{figure}

After dealing with bundles in the next two sections, we will see how to improve upon  lemma~\ref{L: Re-gluing}.

\section{Torus-bundles}\label{S:Torus-bundles}
In this section, we show that each torus-bundle over the circle is the boundary of some Haken $4$--manifold. 

We first introduce some conventions of notation  and orientation that will be used throughout this paper.

If $g \co S \to S$ is a homeomorphism of a surface $S$, then $S(g)$ is the surface-bundle over the circle with fibre $S$ and monodromy $g$. More concretely,
\begin{align}\label{E: notation}
	S(g) = S \times [0,1] / \left(x,0 \right) \sim \left(g(x), 1 \right).
\end{align}
We will use the above notation for fibre-bundles throughout this paper.

The following conventions regarding orientations on manifolds and their boundaries will be used. If $S$ is an orientable surface, then an orientation for $S$ can be specified by an ordered linearly independent pair of vectors $(w,x)$ at a single point $p \in S$. The standard orientation for $S(g)$ is then $(w,x,y)$ where $y$ is a non-zero vector based at $(p,0)$ tangent to $\{p\} \times [0,1]$ and directed towards $1$. The standard orientation  of the $4$--manifold $S(g) \times [0,1]$ is $(w,x,y,z)$ where $z$ is a non-zero vector based at $(p,0,0)$ tangent to $\{(p,0)\} \times [0,1]$ and directed towards $1$. We write the boundary of $S(g) \times [0,1]$ as 
\begin{align}\label{E: orient}
	\bd \left( S(g) \times [0,1] \right) = S(g^{-1}) \sqcup S(g).
\end{align}
Since $S(g^{-1})$ is homeomorphic to $S(g)$, with a reversal of orientation, we use the term $S(g^{-1})$ in expression~(\ref{E: orient}) to represent the manifold $S(g) \times \{0\}$ with the orientation induced by the outward normal convention.   The term $S(g)$ in expression~(\ref{E: orient}) represents the manifold $S(g) \times \{1\}$, also with the outward normal convention.
\begin{eg}\label{Eg:simple bundle}
Let $\torus^2(\varphi)$ be the torus-bundle over a circle with monodromy $\varphi$ a single Dehn twist.
\begin{figure}[ht!]
\centering
 \includegraphics{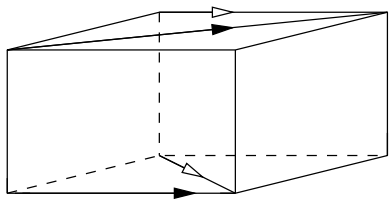}
	\caption{Torus bundle with single Dehn twist.}\label{F:Single twist bundle}
\end{figure}
We represent the torus-bundle $\torus^2(\varphi)$ by considering the torus as the square $[0,1] \times [0,1]$ in the plane with sides identified in the usual way. The monodromy $\varphi$ is represented by the matrix
$\left(
\begin{smallmatrix}
1 & 0 \\
1 & 1
\end{smallmatrix}
\right)$.
We represent $\torus^2(\varphi)$ visually in figure~\ref{F:Single twist bundle}, regarding $\torus^2(\varphi)$ as the quotient space $\left(\mathrm{T}^2 \times [0,1]\right) / (x,0) \sim (\varphi(x),1)$. 

 We consider a special case of lemma~\ref{L: Re-gluing} that we will use subsequently. Let $W_1 = \torus^2(\varphi) \times [0,1]$, which is a torus-bundle over an annulus. The boundary of $W_1$ is $\left( \torus^2(\varphi) \times \{0\} \right) \sqcup \left(\torus^2(\varphi) \times \{1\} \right)$. Let us pick out two disjoint parallel torus fibres in $\torus^2(\varphi) \times \{1\}$.
 These are: $T_i  = \torus^2 \times \{i/3\} \times \{1\}$ for $i  = 1$ or $2$. Let $\varepsilon$ be a sufficiently small positive number\footnote{The number $\varepsilon$ is sufficiently small in the sense that $T_1\left(\varepsilon\right) \cap T_2\left(\varepsilon\right) = \varnothing$.}, and consider the $\varepsilon$--neighbourhoods of these tori: 
$T_i \left(\varepsilon\right) = \mathrm{T^2} \times [i/3 - \varepsilon, i/3 + \varepsilon] \times \{1\}$.
Attach a copy of 
$\mathrm{T}^2 \times [-\varepsilon, \varepsilon] \times [0,1]$
to $\torus^2(\varphi) \times \{1\}$ so that $\torus^2 \times [-\varepsilon, \varepsilon] \times \{0\}$ meets $T_1(\varepsilon)$ and $\torus^2 \times [-\varepsilon, \varepsilon] \times \{1\}$ meets $T_2\left(\varepsilon\right)$. We choose the attachment so that the boundary of the resulting manifold $W$ is
\[
  \torus^2(\varphi^{-1}) \sqcup \torus^2(\psi) \sqcup \torus^2(\varphi \circ \psi)
\]
where $\psi \in \mathrm{SL}(2,\ZZ)$. The manifold $W$ is an orientable Haken $4$--manifold 
with three boundary  components. The orientations on the boundary components is based on the orientation convention in expression~(\ref{E: orient}). If we regard $\psi$ as a product of $k$ Dehn twists, then this example shows how to construct a Haken cobordism between torus-bundles with $k+1$ Dehn twists, $k$ Dehn twists and a single Dehn twist.
\end{eg}

\begin{thm}\label{T: All T-bundles}
If $M$ is a torus-bundle over a circle, then there is a Haken $4$--manifold $W$ with boundary $\bd W = M$.
\end{thm}

We will prove theorem~\ref{T: All T-bundles} via a sequence of lemmas. The first of these is a simple observation that is probably well-known.
\begin{lemma}\label{L:equivalent regluings}
Let $F$ and $G$ be closed orientable incompressible surfaces in a closed orientable $3$--manifold $M$. Suppose that $F \cap G$ is a simple closed curve $\alpha$. The manifold obtained by splitting $M$ open along $F$ and regluing via a Dehn twist along $\alpha$ is homeomorphic to the  manifold obtained by splitting $M$ open along $G$ and regluing via a Dehn twist along $\alpha$.
\end{lemma}
\begin{proof}
The result of either operation is simply Dehn surgery on the curve $\alpha$.
\end{proof}

\begin{lemma}\label{L:single twist}
If $M_\varphi$ is a torus-bundle over a circle with monodromy $\varphi$ a single Dehn twist, then there is a Haken $4$--manifold $W$ with boundary $\bd W =~M_\varphi$.
\end{lemma}

\begin{proof}
Let $\Sigma$ be a closed orientable surface of genus three. We may regard $\Sigma$ as the double of the thrice-punctured disk, which is shown in figure~\ref{F:Lantern relation}.
Three of the four boundary components of the thrice-punctured disk are labelled in figure~\ref{F:Lantern relation}.
Let $\epsilon_i$ be a curve parallel to the boundary component labelled $i$ in figure~\ref{F:Lantern relation}. The curve $\epsilon_4$ is parallel to the unlabelled boundary component. Let $\alpha$, $\beta$ and $\gamma$ be the curves shown in figure~\ref{F:Lantern relation}. 
Up to isotopy, the identity mapping
$\id \co \Sigma \to \Sigma$ can be written as a product of three positive Dehn twists and four negative Dehn twists. This observation is a consequence of the lantern relation \cite{Johnson79} of the mapping class group. 
\begin{figure}[ht!]
\labellist
\pinlabel $\alpha$ at 135 98
\pinlabel $\beta$ at 49 60
\pinlabel $\gamma$ at 182 60
\endlabellist
\centering
 \includegraphics{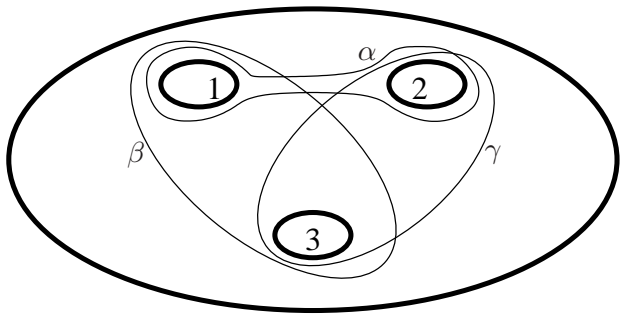}
	\caption{Lantern relation.}\label{F:Lantern relation}
\end{figure}
Let $f_\alpha$ be the right-handed Dehn twist about $\alpha$, and define $f_\beta$ and $f_\gamma$ similarly. Let $f_i$ be the right handed Dehn twist about $\epsilon_i$. The lantern relation is
\[
f_\gamma f_\beta f_\alpha = f_1 f_2 f_3 f_4.
\]
We may write this relation in a number of ways; each $\epsilon_i$ is disjoint from the other curves, so for example, $f_i$ commutes with the other Dehn twists. Thus, up to isotopy, we may write the identity mapping as
\[
  f^{-1}_{1}  f_\gamma f^{-1}_{2} f_\beta  f^{-1}_{3} f_\alpha  f^{-1}_{4}.
\]
We  define the following maps:
\begin{align*}
\theta_7 &= f^{-1}_{1}  f_\gamma f^{-1}_{2} f_\beta  f^{-1}_{3} f_\alpha  f^{-1}_{4} & \theta_3 &= \theta_4 f_\beta^{-1}\\
\theta_6 &= \theta_7 f_4 & \theta_2 &= \theta_3 f_2\\
\theta_5 &= \theta_6 f_\alpha^{-1} & \theta_1 &= \theta_2 f_{\gamma}^{-1}\\
\theta_4 &= \theta_5 f_3 & \theta_0 &= \theta_1 f_1
\end{align*}	

We now show how to construct a Haken $4$--manifold $W_7$ with three surface-bundle boundary components. Specifically,
\[
  \bd W_7 = \Sigma(\theta_7^{-1}) \sqcup \Sigma(\theta_6) \sqcup \torus^2(\varphi).
\] 
The boundary components are written using representatives from the appropriate orientation-preserving homeomorphism class. The orientations of the boundary components are in accordance with the convention from expression~(\ref{E: orient}).

 To see how to build $W_7$, first note that, by the lantern relation, $\theta_7$ is isotopic to the identity, so $\Sigma(\theta_7) = \Sigma \times \sphere^1$. Then observe that $\Sigma \times \sphere^1$ is related to $\Sigma({\theta_6})$ by splitting open along a fibre and regluing by a Dehn twist along the curve $\epsilon_4$ in the fibre. There is an incompressible torus $T$ in $\Sigma \times \sphere^1$ that intersects the fibre in the curve $\epsilon_4$. By lemma~\ref{L:equivalent regluings}, we can also obtain $\Sigma({\theta_6})$ by splitting $\Sigma \times \sphere^1$ open along $T$ and regluing with a Dehn twist. 
 Then lemma~\ref{L: Re-gluing} tells us how to construct $W_7$; we attach a manifold of the form $\torus^2 \times [0,1] \times [0,1]$ to a boundary-component of $\left(\Sigma \times \sphere^1\right) \times [0,1]$.

 Observe that in $\Sigma(\theta_6)$ there is an incompressible torus that intersects the fibre in the curve $\alpha$. By attaching a manifold of the form $\torus^2 \times [0,1] \times [0,1]$ to a boundary-component of $\Sigma(\theta_6) \times [0,1]$ we obtain a Haken $4$--manifold $W_6$ with boundary
\[
  \bd W_6 = \Sigma(\theta_6^{-1}) \sqcup \Sigma(\theta_5) \sqcup \torus^2(\varphi^{-1}).
\]
 
Similarly, there is an incompressible torus in $\Sigma(\theta_5)$ that intersects the fibre in the curve $\epsilon_3$. We then construct a $4$--manifold $W_5$ with boundary
\[
  \bd W_5 = \Sigma(\theta_5^{-1}) \sqcup \Sigma(\theta_4) \sqcup \torus^2(\varphi).
\]
 
We continue creating Haken cobordisms with three boundary components. However, we no longer need to find incompressible tori that intersect the lantern curves. Instead, all the boundary components will be $\Sigma$--bundles over the circle. Lemma~\ref{L: Re-gluing} produces Haken $4$--manifolds $W_4$, $W_3$, $W_2$, and $W_1$ with boundaries as follows:
\begin{align*}
\bd W_4 &= \Sigma(\theta_4^{-1}) \sqcup \Sigma(\theta_3) \sqcup \Sigma(f_\beta^{-1}) &\bd W_2 &= \Sigma(\theta_2^{-1}) \sqcup \Sigma(\theta_1) \sqcup \Sigma(f_\gamma^{-1}) \\
\bd W_3 &= \Sigma(\theta_3^{-1}) \sqcup \Sigma(\theta_2) \sqcup \Sigma(f_2) &\bd W_1 &= \Sigma(\theta_1^{-1}) \sqcup \Sigma(\theta_0) \sqcup \Sigma(f_1) \\
\end{align*}
Note that $\theta_0$ is the identity mapping so $\Sigma(\theta_0) = \Sigma \times \sphere^1$. 

So we have seven orientable Haken $4$--manifolds each with three boundary components. We can glue these seven manifolds together to form a connected manifold $W'$ with boundary
\begin{align*}
\bd W' &= \Sigma(\theta_7^{-1}) \sqcup \torus^2(\varphi) \sqcup \torus^2(\varphi^{-1})\sqcup \torus^2(\varphi) \sqcup \Sigma(f_\beta^{-1})\\
 & \qquad \sqcup \Sigma(f_2) \sqcup \Sigma(f_\gamma^{-1}) \sqcup \Sigma(f_1) \sqcup \Sigma(\theta_0).
\end{align*}
\begin{figure}[ht!]
\labellist
\small\hair 2pt
\pinlabel $T^2(\varphi)$ at 55 103
\pinlabel $T^2(\varphi^{-1})$ at 205 103
\pinlabel $\Sigma(\theta_{7}^{-1})$ at 19 33
\pinlabel $\Sigma(\theta_6)$ at 94 33
\pinlabel $\Sigma(\theta_{6}^{-1})$ at 170 33
\pinlabel $\Sigma(\theta_5)$ at 246 33
\endlabellist
\centering
 \includegraphics{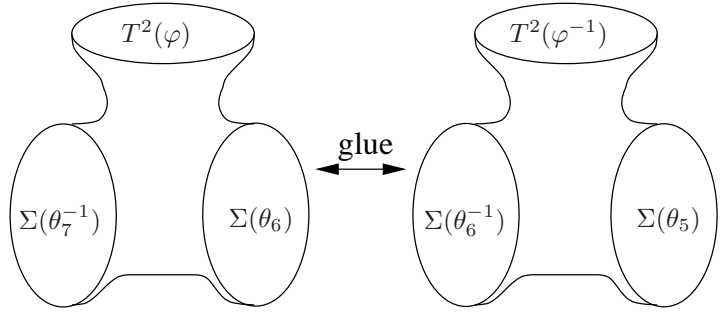}
	\caption{Identifying the $\Sigma(\theta_6)$ boundary-components of $W_7$ and $W_6$ to produce a connected $4$--manifold.}\label{F:CobordFigConnect}
\end{figure}
The idea is illustrated in figure~\ref{F:CobordFigConnect}, which schematically shows the manifolds $W_7$ and $W_6$ being joined together.

We glue eight of these boundary components in pairs, leaving just one boundary component $\torus^2(\varphi)$. That is, we glue $\torus^2(\varphi^{-1})\subset W_6$ to $\torus^2(\varphi) \subset W_5$, glue $\Sigma(f_\beta^{-1})$ to $\Sigma(f_2)$ and glue  $\Sigma(f_\gamma^{-1})$ to $\Sigma(f_1)$.
We also glue $\Sigma(\theta_7)$ to $\Sigma(\theta_0)$.
This can all be done so that the result is orientable. Hence there is an orientable Haken $4$--manifold $W$ with boundary $\bd W = \torus^2(\varphi)$.
\end{proof} 

\begin{lemma}\label{L:T-bundles}
If $M_\psi$ is a torus-bundle over a circle with monodromy $\psi$ a product of a finite number of Dehn twists, then there is a Haken $4$--manifold $W$ with boundary $\bd W = M_\psi$.
\end{lemma}
\begin{proof}
The construction is similar to that of example~\ref{Eg:simple bundle} and is by induction on the number of Dehn twists, say $k$. Write the monodromy as $\psi = \tau \circ \sigma$ where $\tau$ is a product of $k-1$ Dehn twists and $\sigma$ is a Dehn twist. We modify the torus-bundle $M_\psi \times [0,1]$ by attaching a copy of $\torus^2 \times [-\varepsilon, \varepsilon] \times [0,1]$ to $\varepsilon$-neighbourhoods of disjoint torus fibres in $M_\psi \times \{0\}$ as in example~\ref{Eg:simple bundle}, except we choose the gluing so that the boundary components are $M_\psi$, $M_\tau$ and $M_\sigma$. 

Since $\sigma$ is a single Dehn twist, we can glue on the compact $4$--manifold found in lemma~\ref{L:single twist} to fill in the boundary component 
$M_\psi$. We obtain a manifold $W$ with two boundary components $M_\psi, M_\tau$. It is easy to see that $W$ is a Haken $4$--manifold. The proof now follows by induction since $\tau$ is a product of $k-1$ Dehn twists. So we can find a Haken $4$--manifold with boundary $M_\tau$ and glue this onto $W$ to build the required Haken $4$--manifold with boundary $M_\psi$. 

Note that the case $k=1$ follows from lemma~\ref{L:single twist}.
\end{proof}

Putting the results of the lemmas in this section together constitutes a proof of theorem~\ref{T: All T-bundles}.

\section{Higher genus surface-bundles}\label{S:Higher genus surface-bundles}

We will use lemma \ref{L: Re-gluing} in our proof of the main theorem of this section.
\begin{thm}\label{T:surface bundles bound}
If $M$ is a closed surface-bundle over a circle, then there is a Haken $4$--manifold $W$ with $\bd W = M$.
\end{thm}
\begin{proof}
As before, the proof is by induction on the number of Dehn twists needed to represent the monodromy. To prove theorem~\ref{T:surface bundles bound}, we must construct a Haken $4$--manifold whose boundary is a surface-bundle with given monodromy.

To start the induction, let $F$ be a closed orientable surface of genus at least two, and let $M_\varphi$ be the surface bundle $F(\varphi)$,
where $\varphi$ is a Dehn twist along an essential curve $\alpha$ in $F$. It is clear that we can construct $M_\varphi$ from the product $F \times \sphere^1$ by cutting $F \times \sphere^1$ open along the fibre $F \times \{p\}$ and regluing with a Dehn twist. 
By lemma~\ref{L:equivalent regluings}, we can construct $M_\varphi$ by splitting $F \times \sphere^1$ along an incompressible torus containing $\alpha$ and regluing with a Dehn twist. 

The manifold $M_\varphi$ is related to the product manifold $F \times \sphere^1$ by a change in homeomorphism
along an incompressible torus. Hence, there is a Haken $4$--manifold $W_1$ with boundary 
$\bd W_1 = M_\varphi \sqcup (F \times \sphere^1) \sqcup E$ where $E$ is the total space of a torus-bundle 
over a circle. In section \ref{S:Torus-bundles} we showed that $E$ is the boundary of a Haken $4$--manifold, $W_2$. The product $F\times \sphere^1$ is also the boundary of a Haken $4$--manifold. For example, take a Haken $3$--manifold $N$ with boundary $\bd N = F$. Then $N \times \sphere^1$ will suffice. We attach $W_2$ and $N \times \sphere^1$ to the appropriate boundary components of $W_1$ to obtain a Haken $4$--manifold with boundary $M_\varphi$.

To prove the general case, we proceed exactly as in lemma \ref{L:T-bundles}. Assume that $M_\varphi$ is a surface bundle over a circle whose monodromy $\varphi$ is a product of $k$ Dehn twists. Write $\varphi = \tau \circ \psi$ where $\psi$ is a single Dehn twist and $\tau$ is a product of $k-1$ Dehn twists. Using lemma \ref{L: Re-gluing} and the case above of a surface bundle with monodromy consisting of a single Dehn twist, we can construct a Haken $4$--manifold with boundary consisting of the disjoint union of $M_\varphi, M_\tau, M_\psi$ and then glue on a Haken $4$--manifold with boundary $M_\psi$, since $\psi$ is a single Dehn twist. By induction on the number $k$ of Dehn twists, there is another Haken $4$--manifold with boundary $M_\tau$ since $\tau$ is a product of $k-1$ Dehn twists. Gluing this on completes the proof of the theorem. 
\end{proof}

\section{Other Haken manifolds}\label{S:Other Haken manifolds}
We first prove an extension of lemma~\ref{L: Re-gluing}, which gives a sufficient condition for two Haken $3$--manifolds to be Haken cobordant.

\begin{thm}\label{T:HakenCobordism}
If $N$ is obtained from the closed connected Haken $3$--manifold $M$ by splitting $M$ open along an incompressible surface $F$ and re-gluing the boundary components, then there is a Haken $4$--manifold $W$ with $\bd W = M \sqcup N$, and boundary-pattern $\ul w = \{M,N\}$.
\end{thm}
\begin{proof}
We use the construction in the proof of lemma~\ref{L: Re-gluing} to obtain a Haken $4$--manifold $X$ with boundary $\bd X = M \sqcup N \sqcup E$ and boundary-pattern $\ul x = \{M,N,E\}$, where $E$ is a a surface-bundle over a circle with fibre $F$. By theorems~\ref{T: All T-bundles} and~\ref{T:surface bundles bound}, there is another Haken $4$--manifold $Y$ with boundary $\bd Y = E$ and boundary-pattern $\ul y = \{E\}$. We form a quotient space of $X \sqcup Y$ by gluing the $E$ boundary components together via a homeomorphism to obtain the required Haken $4$--manifold $W$.
\end{proof}


Gabai \cite{Gabai} announced the following result in 1983 with an outline of the proof, and recently Ni \cite{Ni} has provided the details of the proof. 
\begin{thm}\label{T: Gabai-Ni}
Let $M_1$ be a closed Haken $3$--manifold. There is a sequence
\[
M_1, M_2, M_3, \dots, M_n
\]
such that $M_{i+1}$ is obtained from $M_i$ by splitting $M_i$ open along an incompressible surface and re-gluing the boundary components, and $M_n$ is a product $\Sigma \times \sphere^1$, where $\Sigma$ is a closed surface.
\end{thm}
Using theorem~\ref{T: Gabai-Ni}, we can show that every pair of closed Haken $3$--manifolds is the boundary of some Haken $4$--manifold.
\begin{thm}
Let $M, M^\prime$ be a pair of closed Haken $3$--manifolds. Then there is a Haken $4$--manifold $W$ with $\bd W = M \sqcup M^\prime$ and boundary-pattern $\ul w = \{M, M^\prime\}$.
\end{thm}
\begin{proof}
Write $M = M_1$ and using the notation of theorem~\ref{T: Gabai-Ni} we have a sequence
\[
M_1, M_2, M_3, \dots, M_n
\]
such that $M_{i+1}$ is obtained from $M_i$ by splitting $M_i$ open along an incompressible surface and re-gluing the boundary components, and $M_n = \Sigma \times \sphere^1$, for some closed orientable aspherical surface $\Sigma$. Using induction on the number of terms in the sequence, we use theorem~\ref{T:HakenCobordism} to obtain a Haken $4$--manifold $X$ with boundary $\bd X = M_1 \sqcup M_n$. Similarly (with obvious notation) there is a Haken $4$--manifold $Y$ with boundary  $\bd Y = M_1^\prime \sqcup M_p^\prime$, where $M_p^\prime = \Sigma^\prime \times \sphere^1$ and $\Sigma^\prime$ is a closed orientable aspherical surface. If $\Sigma^\prime$ is homeomorphic to $\Sigma$ we can glue $X$ to $Y$ along the product boundary components to obtain the required Haken cobordism. Otherwise, take a Haken $3$--manifold $N$ with boundary $\bd N = \Sigma \sqcup \Sigma^\prime$. Then $N \times \sphere^1$ is a Haken $4$--manifold with boundary $(\Sigma \times \sphere^1) \sqcup (\Sigma^\prime \times \sphere^1)$. We can then glue $X$ and $Y$ to the appropriate boundary components of $N \times \sphere^1$ to obtain the required Haken cobordism. 
\end{proof}

\begin{cor}\label{C:HakenBoundary}
If $M$ is a closed Haken $3$--manifold, then there is a Haken $4$--manifold $W$ with $\bd W = M$ and boundary-pattern $\ul w = \{M\}$.
\end{cor}

\section{Hyperbolic case}
In Long and Reid \cite{Long2000}, it is shown that if a closed hyperbolic $3$--manifold $M$ is the totally geodesic boundary of a compact hyperbolic $4$--manifold $W$, then $\eta(M)$ takes an integer value. In \cite{Long2000} $M$ is said to \emph{geometrically bound} $W$. On the other hand, Meyerhoff and Neumann \cite{Meyer92}, show that $\eta(N_\alpha)$ takes a dense set of values in $\RR$ for the set $\{N_\alpha\}$ of Dehn surgeries on a hyperbolic knot in $S^3$. So this implies that `generically' hyperbolic $3$--manifolds do not geometrically bound hyperbolic $4$-manifolds. 

The existence of $\pi_1$--injective $2$-tori in the Haken $4$--manifolds constructed in Corollary~\ref{C:HakenBoundary} gives an obvious obstruction to these $4$--manifolds admitting hyperbolic or even strictly negatively curved metrics. 

In \cite{Long2001} Long and Reid give examples of $n$--dimensional hyperbolic manifolds which geometrically bound hyperbolic $(n+1)$--dimensional hyperbolic manifolds, for all $n$. 

\section{Some questions}\label{S:Some questions}
The Haken $4$--manifolds that we have constructed in this paper fall into a special class. In a sense, they are analogues of the graph manifolds of Waldhausen. Other examples of Haken $4$--manifolds exist. For example, the hyperbolic $4$--manifolds of Ratcliffe and Tschantz \cite{Rat-Tsch} are all finitely covered by Haken $4$--manifolds (see \cite{FoozRubin2011} for a proof of this). In \cite{FoozRubin2011}, examples of Haken $4$--manifolds which admit metrics of strictly negative curvature but which do not admit hyperbolic metrics are given. 

\begin{question}
If $M$ is a closed Haken $3$--manifold, does there exist a Haken $4$--manifold $W$ with $\bd W = M$ and which contains only non-separating submanifolds in its hierarchy? (Note that then the complement of the hierarchy is a single $4$--cell.)
\end{question}

\begin{question}
Which closed hyperbolic Haken $3$--manifolds $M$ geometrically bound hyperbolic Haken $4$--manifolds? Are there other obstructions than that in \cite{Long2000} that the eta invariant of $M$ must be an integer? What about the situation if the Haken $4$--manifold admits a metric of strictly negative or non-positive curvature? Is it still true that the eta invariant of $M$ must be an integer in this case?
\end{question}

\begin{question}
For $n >3$, what are the standard cobordism classes for Haken $n$--manifolds? We say that Haken $n$-manifolds $N$ and $N'$ belong to the same standard cobordism class if there is an $(n + 1)$--manifold $W$ for which $\bd W = N \sqcup N'$.  In private communication, Allan Edmonds has contructed a Haken $4$--manifold with odd Euler characteristic, so we know, for example, that Haken $4$--manifolds need not be null cobordant.
\end{question}
%
%
%
\bibliographystyle{amsplain}
\bibliography{H4Mbiblio}

\providecommand{\bysame}{\leavevmode\hbox to3em{\hrulefill}\thinspace}
\providecommand{\MR}{\relax\ifhmode\unskip\space\fi MR }
\providecommand{\MRhref}[2]{%
  \href{http://www.ams.org/mathscinet-getitem?mr=#1}{#2}
}
\providecommand{\href}[2]{#2}
\begin{thebibliography}{10}

\bibitem{DavisJan}
Michael Davis and Tadeusz Januszkiewicz, \emph{Hyperbolization of polyhedra},
  J. Diff. Geom. \textbf{34} (1991), 347--388.

\bibitem{DavisJanWein}
Michael Davis, Tadeusz Januszkiewicz, and Shmuel Weinberger, \emph{Relative
  hyperbolization and aspherical bordisms; an addendum to `hyperbolization of
  polyhedra'}, J. Diff. Geom. \textbf{58} (2001), 535--541.

\bibitem{Fooz}
Bell Foozwell, \emph{{H}aken $n$--manifolds}, PhD Thesis, University of
  Melbourne (2007), available at
  \texttt{https://sites.google.com/site/bellfoozwell/}.

\bibitem{Fooz1}
\bysame, \emph{The universal covering space of a {H}aken $n$--manifold},
  preprint (2011), available at \texttt{arXiv:1108.0474}.

\bibitem{FoozRubin2011}
Bell Foozwell and Hyam Rubinstein, \emph{Introduction to the theory of {H}aken
  $n$--manifolds}, Topology and Geometry in Dimension Three: Triangulations,
  Invariants, and Geometric Structures, American Mathematical Society,
  Providence, Rhode Island, 2011, pp.~71--84.

\bibitem{Gabai}
David Gabai, \emph{An internal hierarchy for {$3$}-manifolds}, Knot theory and
  Manifolds, Lecture Notes in Math., vol. 1144, Springer, Providence, R.I.,
  1983, pp.~14--18.

\bibitem{Johannson}
Klaus Johannson, \emph{Homotopy equivalences of $3$--manifolds with
  boundaries}, Lecture Notes in Math., vol. 761, Springer, Berlin, 1979.

\bibitem{Johnson79}
Dennis~L. Johnson, \emph{Homeomorphisms of a surface which act trivially on
  homology}, Proc. Amer. Math. Soc. \textbf{75} (1979), 119--125.

\bibitem{Long2000}
Darren Long and Alan Reid, \emph{On the geometric boundaries of hyperbolic
  $4$-manifolds}, Geom. Topol. \textbf{4} (2000), 171--178.

\bibitem{Long2001}
\bysame, \emph{Constructing hyperbolic manifolds which bound geometrically},
  Math. Research Letters \textbf{8} (2001), 443--456.

\bibitem{Meyer92}
Robert Meyerhoff and Walter Neumann, \emph{an asymptotic formula for the eta
  invariant for hyperbolic $3$--manifolds}, Comm. Math. Helv. \textbf{67}
  (1992), 28--46.

\bibitem{Ni}
Yi~Ni, \emph{Some applications of {G}abai's internal hierarchies}, preprint
  (2011), available at \texttt{{arXiv:1111.0914}}.

\bibitem{Rat-Tsch}
John Ratcliffe and Steven Tschantz, \emph{The volume spectrum of hyperbolic
  $4$-manifolds}, Exp. Math. \textbf{9} (2000), 101--125.

\bibitem{Wald}
Friedhelm Waldhausen, \emph{The word problem in fundamental groups of
  sufficiently large irreducible {$3$}-manifolds}, Ann. of Math. (2)
  \textbf{88} (1968), 272--280.

\end{thebibliography}

\end{document}